\documentclass[12pt]{amsart}
\usepackage{amsfonts}
\usepackage{ifthen}
\usepackage{verbatim}
\usepackage{amsthm}
\usepackage{amsmath}
\usepackage{amssymb}
\usepackage{graphicx}
\usepackage{subfigure}
\usepackage{amscd,amssymb,amsthm}
\usepackage{color,xcolor}

\newcounter{minutes}
\divide\time by 60
\newcounter{hours}
\multiply\time by 60 \addtocounter{minutes}{-\time}

\title{Lipschitz conditions and the distance ratio metric }

\author{Slavko Simi\'c}
\author{Matti Vuorinen}

\address{ Mathematical Institute SANU, Kneza Mihaila 36, 11000
Belgrade, Serbia} \email{ ssimic@turing.mi.sanu.ac.rs}
\address{Department of Mathematics and Statistics, University of Turku, 20014 Turku,
Finland} \email{vuorinen@utu.fi}


\newtheorem{theorem}[equation]{Theorem}

\newtheorem{lemma}[equation]{Lemma}

\newtheorem{nonsec}[equation]{}

\newcommand{\beq}{\begin{equation}}
\newcommand{\eeq}{\end{equation}}
\newcommand{\B}{\mathbb{B}^2}

\newcommand{\Rn}{ {\mathbb{R}^n} }

\numberwithin{equation}{section}

\begin{document}

\def\thefootnote{}
\footnotetext{ \texttt{\tiny File:~\jobname .tex,
          printed: \number\year-\number\month-\number\day,
          \thehours.\ifnum\theminutes<10{0}\fi\theminutes}
} \makeatletter\def\thefootnote{\@arabic\c@footnote}\makeatother


\maketitle

\begin{abstract}
We give study the
Lipschitz continuity of M\"obius transformations of a punctured disk
onto another punctured disk
with respect to the distance ratio metric.
\end{abstract}

\pagestyle{plain}

\section{Introduction}

During the past thirty years the theory of quasiconformal maps has been studied
in various contexts such as in
Euclidean, Banach, or even metric spaces. It has turned out that while some classical tools based on conformal invariants, real analysis and measure theory are no longer useful beyond the Euclidean context, the notion of a metric space and related notions
still provide a useful conceptual
framework. This has led to the study of the geometry defined
by various metrics and to the key role of metrics in recent
theory of quasiconformality. See e.g. \cite{ccq,himps,hpww,k,rt,rt2,free}.

{\bf Distance ratio metric.} \label{drm1.8} One of these metrics
is the distance ratio metric.
For a subdomain $G
\subset {\mathbb R}^n\,$ and for $x,y\in G$ the {\em distance
ratio metric $j_G$} is defined by
\begin{equation} \label{1.1}
 j_G(x,y)=\log \left( 1+\frac{|x-y|}{\min \{d_G(x),d_G(y) \} } \right)\,,
\end{equation}
where $d_G(x)$ denotes the Euclidean distance from $x$ to
$\partial G$. If $G_1 \subset G$ is a proper subdomain then for $x,y \in G_1 $ clearly
\begin{equation} \label{1.1b} 
j_G(x,y) \le j_{G_1}(x,y)  \,.
\end{equation}
Moreover, the numerical value of the metric is highly sensitive to
boundary variation, the left and right sides of \eqref{1.1b} are not
comparable even if $G_1= G \setminus \{p\}, p \in G \,.$

The distance ratio metric was introduced by F.W. Gehring and B.P.
Palka \cite{gp} and in the above, simplified, form by M. Vuorinen \cite{vu85}
and it is frequently used in the study of hyperbolic type metrics
\cite{himps} and geometric theory of functions. It is a basic fact that the above
$j$-metric is closely related to the hyperbolic metric both for the unit
ball ${\mathbb B}^n \subset {\mathbb R}^n$ and for the Poincar\'e half-space ${\mathbb H}^n\,,$ \cite{v}.

{\bf Quasi-invariance of $j_G$.} \label{qinv}

Given domains $G,G' \subset {\mathbb R}^n$ and an open continuous mapping $f: G \to G'$ with $fG \subset G'$ we consider the
following condition: there exists a constant $C \ge 1$ such that
for all $x,y\in G$ we have
\begin{equation} \label{jlip}
j_{G'}(f(x), f(y)) \le C j_G(x,y) \,,
\end{equation}
or, equivalently, that the mapping
$$  f:(G, j_G) \to (G',j_{G'}) $$
between metric spaces is Lipschitz continuous with the Lipschitz
constant $C\,.$

The hyperbolic metric in the unit ball or half space
are M\"obius invariant.
However,
the distance
ratio metric $j_G$ is not invariant under M\"obius transformations.
Therefore, it is natural to ask what the Lipschitz constants are
for these metrics under conformal mappings or M\"obius
transformations in higher dimension. F.\,W. Gehring 
and B.\,G. Osgood proved that these metrics are not changed by
more than a factor $2$ under M\"obius transformations, see
\cite[proof of Theorem 4]{go}:

\begin{theorem} \label{gplem}
If $D$ and $D'$ are proper subdomains of $\Rn$ and if $f$ is a
M\"obius transformation of $D$ onto $D'$, then for all $x,y\in D$
$$\frac12 j_{D}(x,y)\leq j_{D'}(f(x),f(y))\leq 2j_D(x,y).$$

\end{theorem}

It is easy to see that for a M\"obius transformation $f : {\mathbb
B}^n \to  {\mathbb B}^n$ with $f(0) \neq 0,$  and $ x,y \in
{\mathbb B}^n, x\ne y,$ the $j_{ {\mathbb B}^n}$ distances need
not be the same. On the other hand, the next theorem from \cite{svw}, conjectured in \cite{kvz}, yields a sharp form of Theorem \ref{gplem} for
M\"obius automorphisms of the unit ball.

\begin{theorem} \label{kvzconj} A M\"obius transformation
$f: {\mathbb B}^n \to {\mathbb B}^n= f({\mathbb B}^n)$ satisfies
$$
j_{{\mathbb B}^n}(f(x),f(y)) \le (1+ |f(0)|) j_{{\mathbb
B}^n}(x,y)
$$
for all $x,y \in {\mathbb B}^n \,.$ The constant is best possible.
\end{theorem}

\bigskip

A similar result for a punctured disk was conjectured in \cite{svw}. The next theorem, our main result, settles this
conjecture from  {\cite{svw}} in the affirmative.

\begin{theorem} \label{mainthm}
Let $a\in\B$ and $h: \B\setminus \{0\}\rightarrow\B\setminus
\{a\}$ be a M\"obius transformation with $h(0)=a$. Then for
$x,\,y\in \B\setminus \{0\}$
\begin{eqnarray*}
j_{\B\setminus \{a\}}(h(x),h(y))\le C(a) j_{\B\setminus
\{0\}}(x,y),
\end{eqnarray*}
where the constant $ C(a)=1+(\log\frac{2+|a|}{2-|a|})/\log3$ is
best possible.
\end{theorem}

\bigskip

Clearly the constant $C(a)<1+|a| <2$ for all $a \in \B$ and hence
the constant in Theorem \ref{mainthm} is smaller than the constant
$1+|f(0)|$ in Theorem \ref{kvzconj} and far smaller than the
constant $2$ in Theorem \ref{gplem}.

If $a=0$ in Theorem \ref{mainthm}, then $h$ is a rotation of the unit
disk and hence a Euclidean isometry. Note that $C(0)=1\,,$ i.e. the
result is sharp in this case.

The proof is based on Theorem \ref{s1} below and on Lemma \ref{Le2}, a monotone form of l'H${\rm \hat{o}}$pital's rule from \cite[Theorem 1.25]{avv}.


\section{Preliminary results}

In view of the definition of the distance ratio metric it is
natural to expect that some properties of the logarithm will be
needed. In the earlier paper \cite{svw}, the classical Bernoulli
inequality \cite[(3.6)]{v} was applied for this purpose.
Apparently now some other inequalities are needed and we use the
following result, which is precise and allows us to get rid of
logarithms in further calculations.

\bigskip

\begin{theorem}\label{s1}
Let $D$ and $D'$ be proper subdomains of $\Rn$. For an open
continuous mapping $f: D \to D'$ denote
\small
$$
X=X(z,w):=\frac{|z-w|}{\min \{d_D(z),d_D(w)\}}; \ \
Y=Y(z,w):=\frac{|z-w|}{|f(z)-f(w)|}\frac{\min
\{d_{D'}(f(z)),d_{D'}((f(w))}{\min \{d_D(z),d_D(w)\}}.
$$
\normalsize
If there exists $q, \ 0\le q\le 1$ such that
\begin{equation}\label{eq1}
q\le Y+\frac{Y-1}{X+1},
\end{equation}
then the inequality

$$
j_{D'}(f(z),f(w))\le \frac{2}{1+q}j_{D}(z,w),
$$
holds for all $z,w\in D$.

\end{theorem}

\begin{proof}

The proof is based on the following assertion.

\begin{lemma}\label{Le1}

For $ a\ge 0, q\in [0,1]$, we have

$$
\log\Bigl(\frac{q+e^a}{1+qe^a}\Bigr)\le \frac{1-q}{1+q}a.
$$
\end{lemma}

\begin{proof}
Denote
$$
f(a,q):= \log\Bigl(\frac{q+e^a}{1+qe^a}\Bigr)- \frac{1-q}{1+q}a.
$$
By differentiation, we have

$$
f'_a(a,q)=-\frac{q(1-q)}{1+q}\frac{(e^a-1)^2}{(1+qe^a)(q+e^a)},
$$
we conclude that
$$
f(a,q)\le f(0,q)=0.
$$
\end{proof}

Now, since
$$
X=\frac{|z-w|}{\min \{d_D(z),d_D(w) \}}=\exp(j_D(z,w))-1,
$$

and

$$
Y=\frac{|z-w|}{|f(z)-f(w)|}\frac{\min \{d_{D'}(f(z)),d_{D'}((f(w))
\}}{\min \{d_D(z),d_D(w)
\}}=\frac{\exp(j_D(z,w))-1}{\exp(j_{D'}(f(z),f(w)))-1},
$$

the condition \eqref{eq1} is equivalent to
$$
\exp(j_{D'}(f(z),f(w)))\le
\exp(j_D(z,w))\Bigl(\frac{q+e^{j_D(z,w)}}{1+qe^{j_D(z,w)}}\Bigr).
$$
Therefore, by Lemma \ref{Le1}, we get

$$
j_{D'}(f(z),f(w))\le
j_D(z,w)+\log\Bigl(\frac{q+e^{j_D(z,w)}}{1+qe^{j_D(z,w)}}\Bigr)
$$
$$
\le j_D(z,w)+\frac{1-q}{1+q}j_D(z,w)=\frac{2}{1+q}j_D(z,w).
$$
\end{proof}

In the sequel we shall need the so-called {\em monotone form of
l'H${\rm \hat{o}}$pital's rule}.

\begin{lemma} \label{Le2}{\rm \cite[Theorem 1.25]{avv}}.
For $-\infty<a<b<\infty$, let $f,\,g: [a,b]\rightarrow \mathbb{R}$
be continuous on $[a,b]$, and be differentiable on $(a,b)$, and
let $g'(x)\neq 0$ on $(a,b)$. If $f'(x)/g'(x)$ is
increasing(deceasing) on $(a,b)$, then so are
\begin{eqnarray*}
\frac{f(x)-f(a)}{g(x)-g(a)}\,\,\,\,\,\,\,and\,\,\,\,\,\,\,\,\frac{f(x)-f(b)}{g(x)-g(b)}.
\end{eqnarray*}
If $f'(x)/g'(x)$ is strictly monotone, then the monotonicity in
the conclusion is also strict.
\end{lemma}

\bigskip
Lemma \ref{Le2} has found numerous applications recently. See the bibliography of \cite{avz} for a long list of applications to inequalities.
\bigskip

\begin{lemma}\label{l3}
For positive numbers $A,B,D$ and $0<C<1, \theta\ge 0$, we have

1. The inequality

$$
1+\frac{B}{D}\theta(1+\frac{D}{1+A})(1+\frac{B}{1-C}\theta)\le
(1+\frac{B}{D}\theta)(1+\frac{B}{1-C}\theta),
$$

holds if and only if $B\theta\le A+C$;

2. The function

$$
\frac{\log(1+\frac{B}{1-C}\theta)}{\log(1+\frac{B}{D}\theta)}
$$

is monotone increasing (decreasing) in $\theta$ if $C+D<1 \
(C+D>1)$.

\end{lemma}

\begin{proof}
Proof of the first part follows by direct calculation.

For the second part, set

$$
f_1(\theta)=\log(1+\frac{B}{1-C}\theta) , f_1(0)=0; \ \
f_2(\theta)=\log(1+\frac{B}{D}\theta), f_2(0)=0.
$$
Since

$$
\frac{f'_1(\theta)}{f'_2(\theta)}=\frac{D+B\theta}{1-C+B\theta}=1+\frac{C+D-1}{1-C+B\theta},
$$
the proof follows according to Lemma \ref{Le2}.

\end{proof}

\bigskip

\section{Proof of Theorem \ref{mainthm}}

\begin{nonsec}{\bf Proof of Theorem \ref{mainthm}.} \end{nonsec}
For the proof, define $h(z)=\frac{z+a}{1+\bar a z}$ and suppose in
the sequel that $|z|\ge |w|$. Then

$$
j_G(z,w)=\log\Bigl(1+\frac{|z-w|}{\min\{|z|,|w|,1-|z|,1-|w|\}}\Bigr)
=\log\Bigl(1+\frac{|z-w|}{\min\{|w|,1-|z|\}}\Bigr),
$$
and
$$
j_{G'}(h(z),h(w))=\log\Bigr(1+\frac{|h(z)-h(w)|}{T}\Bigr),
$$
where
$$
T=T(a,z,w):=\min\{ |h(z)-a|,|h(w)-a|,1-|h(z)|,1-|h(w)|\}.
$$

In concert with the definition of the number $T,$ the proof is
divided into four cases. We shall consider each case separately
applying Bernoulli inequality in the first case, its stronger form
from Theorem \ref{s1} in the second one and a direct approach in
the last two cases.

\bigskip

1. \ $T=|h(z)-a|$.

\bigskip

Since $|h(z)-a|=\frac{(1-|a|^2)|z|}{|1+\bar a z|}$ and
$|h(z)-h(w)|=\frac{(1-|a|^2)|z-w|}{|1+\bar a z||1+\bar a w|}$, we
have
$$
j_{G'}(h(z),h(w))=\log\Bigr(1+\frac{|z-w|}{|z||1+\bar a w|}\Bigr).
$$
Suppose firstly that $|w|\le 1-|z|$. Since also $|w|\le 1-|z|\le
1-|w|$, we conclude that $0\le |w|\le 1/2$. Hence, by the Bernoulli inequality (see e.g. \cite[(3.6)]{v}), we get$$
j_{G'}(h(z),h(w))\le\log\Bigl(1+\frac{|z-w|}{|z|(1-|a|
|w|)}\Bigr)\le\log\Bigl(1+\frac{|z-w|}{|w|(1-\frac
{|a|}{2})}\Bigr)
$$
$$
\le\frac{1}{1-\frac
{|a|}{2}}\log\Bigl(1+\frac{|z-w|}{|w|}\Bigr)=\frac{1}{1-\frac
{|a|}{2}}j_G(z,w).
$$

\bigskip

Suppose now $1-|z|\le |w|(\le |z|)\,.$ Then $1/2\le |z|<1$.

\bigskip

Since in this case $(|z|-\frac{1}{2})(2-|a|(1+|z|))\ge 0$, we
easily obtain that
$$
\frac{1}{|z|(1-|a||z|)}\le\frac{1}{(1-\frac {|a|}{2})(1-|z|)}.
$$
Hence,
$$
j_{G'}(h(z),h(w))\le\log\Bigl(1+\frac{|z-w|}{|z|(1-|a||w|)}\Bigr)
\le\log\Bigl(1+\frac{|z-w|}{|z|(1-|a||z|)}\Bigr)
$$
$$
\le\log\Bigl(1+\frac{|z-w|}{(1-\frac {|a|}{2})(1-|z|)}\Bigr)\le
\frac{1}{1-\frac
{|a|}{2}}\log\Bigl(1+\frac{|z-w|}{1-|z|}\Bigr)=\frac{1}{1-\frac
{|a|}{2}}j_G(z,w).
$$

\bigskip

2. \ $T=|h(w)-a|$.

\bigskip

This case can be treated by means of Theorem \ref{s1} with the
same resulting constant $C_1(a)=\frac{2}{2-|a|}$.

Indeed, in terms of Theorem \ref{s1}, we consider firstly the case
$|w|\le 1-|z|$.

We get

$$
X=\frac{|z-w|}{|w|}\ge \frac{|z|-|w|}{|w|}= \frac{|z|}{|w|}-1=X^*,
$$
and

$$
Y=|1+\bar a z|\ge 1-|a||z|=Y^*.
$$
Therefore,

$$
Y+\frac{Y-1}{X+1}\ge
Y^*-\frac{1-Y^*}{1+X^*}=1-|a||z|-|a|||z|\frac{|w|}{|z|}
$$
$$
=1-|a|(|w|+|z|)\ge 1-|a|=q.
$$

\bigskip

In the second case, i.e. when $1-|z|\le |w|$, we want to show that
$$
Y+\frac{Y-1}{X+1}\ge 1-|a|\,, \,{\rm with}\,X=\frac{|z-w|}{1-|z|}\,, Y=|1+\bar{a}z|\frac{|w|}{1-|z|}\,.
$$
This is equivalent to
$$
(Y-(1-|a|))(1+X)+Y\ge 1.
$$
Since in this case
$$
X=\frac{|z-w|}{1-|z|}\ge\frac{|z|-|w|}{1-|z|}:=X^*
$$
and

$$
Y=|1+\bar{a}z|\frac{|w|}{1-|z|}\ge
(1-|a||z|)\frac{|w|}{1-|z|}=1-|a||z|+(|w|+|z|-1)\frac{1-|a||z|}{1-|z|}:=Y^*,
$$
we get

$$
(Y-(1-|a|))(1+X)+Y-1\ge(Y^*-(1-|a|))(1+X^*)+Y^*-1
$$
$$
=\Bigl[|a|(1-|z|)+(|w|+|z|-1)\frac{1-|a||z|}{1-|z|}\Bigr]\frac{1-|w|}{1-|z|}-|a||z|+(|w|+|z|-1)\frac{1-|a||z|}{1-|z|}
$$
$$
\ge
|a|(1-|w|-|z|)+(|w|+|z|-1)\frac{1-|a||z|}{1-|z|}=(|w|+|z|-1)\frac{1-|a|}{1-|z|}\ge
0.
$$

\bigskip

Therefore by Theorem \ref{s1}, in both cases we get

$$
j_{G'}(h(z),h(w))\le \frac{2}{1+q}j_{G}(z,w)=
\frac{2}{2-|a|}j_{G}(z,w)=C_1(a)j_{G}(z,w).
$$

\bigskip

3. \ $T=1-|h(z)|$.

\bigskip

In this case, applying well-known assertions

$$
|1+\bar{a}z|^2-|a+z|^2=(1-|a|^2)(1-|z|^2); \ \
|h(z)|\le\frac{|a|+|z|}{1+|a||z|},
$$
and

$$
|1+\bar{a}w|\ge 1-|a||w|(\ge 1-|a||z|),
$$
we get

$$
j_{G'}(h(z),h(w))=
\log\Bigl(1+\frac{|z-w|(1-|a|^2)}{|1+\bar{a}w|(|1+\bar{a}z|-|z+a|)}\Bigr)
=\log\Bigl(1+\frac{|z-w|(|1+\bar{a}z|+|z+a|)}{|1+\bar{a}w|(1-|z|^2)}\Bigr)
$$
$$
=\log\Bigl(1+\frac{|z-w|}{1-|z|^2}|1+\frac{\bar{a}(z-w)}{1+\bar{a}w}|
(1+\frac{|z+a|}{|1+\bar{a}z|})\Bigr)\le\log\Bigl(1+\frac{|z-w|}{1-|z|}(1+\frac{|a||z-w|}{1-|a||w|})
(1+\frac{|a|(1-|z|)}{1+|a||z|})\Bigr).
$$

Applying here Lemma \ref{l3}, part 1., with
$$
A=|a||z|, \ B=|a|, \ C=|a||w|, D=|a|(1-|z|), \ \theta=|z-w|,
$$
we obtain

\begin{equation}\label{eq2}
j_{G'}(h(z),h(w))\le
\log\Bigl[\Bigl(1+\frac{|z-w|}{1-|z|}\Bigr)\Bigl(1+\frac{|a||z-w|}{1-|a||w|}\Bigr)\Bigr].
\end{equation}

Suppose that $1-|z|\le |w| (\le |z|)$. By Lemma \ref{l3}, part 2.,
\textcolor{black}{ with
$$
B=|a|, \ C=|a||z|, \ D=|a|(1-|z|), \ \theta=|z-w|,
$$}
we get
$$
J(z,w;a):=\frac{j_{G'}(h(z),h(w))}{j_{G}(z,w)}\le
1+\frac{\log(1+\frac{|a||z-w|}{1-|a||z|})}{\log\Bigl(1+\frac{|z-w|}{1-|z|}\Bigr)}
$$
$$
\le
1+\frac{\log(1+\frac{2|a||z|}{1-|a||z|})}{\log\Bigl(1+\frac{2|z|}{1-|z|}\Bigr)},
$$
because in this case we have $C+D=|a|<1$ and $|z-w|\le 2|z|$.

\bigskip

Since the last function is monotone decreasing in $|z|$ and
$|z|\ge 1/2$, we obtain

$$
J(z,w;a)\le
1+\frac{\log(\frac{1+\frac{1}{2}|a|}{1-\frac{1}{2}|a|})}{\log\Bigl(\frac{1+\frac{1}{2}}{1-\frac{1}{2}}\Bigr)}
=1+(\log\frac{2+|a|}{2-|a|})/\log 3:=C_2(a).
$$

\bigskip

Let now $|w|\le 1-|z|(\le 1-|w|)$. The estimation \eqref{eq2} and
Lemma \ref{l3}, part 2.,
\textcolor{black} { with
$$
B=|a|, \ C=D=|a||w|,  \ \theta=|z-w|,
$$}
 yield

$$
J(z,w;a)\le
\frac{\log\Bigl[\Bigl(1+\frac{|z-w|}{1-|z|}\Bigr)\Bigl(1+\frac{|a||z-w|}{1-|a||w|}\Bigr)\Bigr]}{\log\Bigl(1+\frac{|z-w|}{|w|}\Bigr)}
\le
\frac{\log\Bigl[\Bigl(1+\frac{|z-w|}{|w|}\Bigr)\Bigl(1+\frac{|a||z-w|}{1-|a||w|}\Bigr)\Bigr]}{\log\Bigl(1+\frac{|z-w|}{|w|}\Bigr)}
$$
$$
=1+\frac{\log\Bigl(1+\frac{|a||z-w|}{1-|a||w|}\Bigr)}{\log\Bigl(1+\frac{|z-w|}{|w|}\Bigr)}\le
1+\frac{\log\Bigl(1+\frac{|a|}{1-|a||w|}\Bigr)}{\log\Bigl(1+\frac{1}{|w|}\Bigr)},
$$
since $C+D=2|a||w|\le|a|<1$ and $0\le |z-w|\le |z|+|w|\le 1$.

\bigskip

Denote the last function as $g(|w|)$ and let $|w|=r, \ 0<r\le
1/2$. Since

$$
g'(r)=\frac{|a|^2}{(1-r|a|)(1+(1-r)|a|)\log(1+1/r)}+\frac{\log\Bigl(1+\frac{|a|}{1-|a|r}\Bigr)}{r(1+r)\log^2(1+1/r)}>0,
$$
we finally obtain

 $$
J(z,w;a)\le
1+\frac{\log\Bigl(1+\frac{|a|}{1-|a|/2}\Bigr)}{\log(1+2)} =C_2(a).
 $$

\bigskip

4. \ $T=1-|h(w)|$.

\bigskip

This case can be considered analogously with the previous one.

$$
j_{G'}(h(z),h(w))=
\log\Bigl(1+\frac{|z-w|(1-|a|^2)}{|1+\bar{a}z|(|1+\bar{a}w|-|w+a|)}\Bigr)
=\log\Bigl(1+\frac{|z-w|(|1+\bar{a}w|+|w+a|)}{|1+\bar{a}z|(1-|w|^2)}\Bigr)
$$
$$
=\log\Bigl(1+\frac{|z-w|}{1-|w|^2}|1+\frac{\bar{a}\textcolor{red}{ (w-z)}}{1+\bar{a}z}|
(1+\frac{|w+a|}{|1+\bar{a}w|})\Bigr)\le\log\Bigl(1+\frac{|z-w|}{1-|w|}(1+\frac{|a||z-w|}{1-|a||z|})
(1+\frac{|a|(1-|w|)}{1+|a||w|})\Bigr).
$$

Applying Lemma \ref{l3}, part 1., with
$$
A=|a||w|, \ B=|a|, \ C=|a||z|, D=|a|(1-|w|), \ \theta=|z-w|,
$$
we obtain

\begin{equation}\label{eq3}
j_{G'}(h(z),h(w))\le
\log\Bigl[\Bigl(1+\frac{|z-w|}{1-|w|}\Bigr)\Bigl(1+\frac{|a||z-w|}{1-|a||z|}\Bigr)\Bigr].
\end{equation}

Suppose that $1-|z|\le |w| (\le |z|)$. We get
$$
J(z,w;a):=\frac{j_{G'}(h(z),h(w))}{j_{G}(z,w)}\le
\frac{\log\Bigl[\Bigl(1+\frac{|z-w|}{1-|z|}\Bigr)\Bigl(1+\frac{|a||z-w|}{1-|a||z|}\Bigr)\Bigr]}{\log\Bigl(1+\frac{|z-w|}{1-|z|}\Bigr)}
$$
$$
=1+\frac{\log(1+\frac{|a||z-w|}{1-|a||z|})}{\log\Bigl(1+\frac{|z-w|}{1-|z|}\Bigr)}
$$
and this inequality is already considered above.

\bigskip

In the case $|w|\le1-|z|\le1-|w|$, we have

$$
J(z,w;a)\le\frac{\log\Bigl[\Bigl(1+\frac{|z-w|}{1-|w|}\Bigr)\Bigl(1+\frac{|a||z-w|}{1-|a||z|}\Bigr)\Bigr]}{\log\Bigl(1+\frac{|z-w|}{|w|}\Bigr)}
\le\frac{\log\Bigl[\Bigl(1+\frac{|z-w|}{|w|}\Bigr)\Bigl(1+\frac{|a||z-w|}{1-|a|(1-|w|)}\Bigr)\Bigr]}{\log\Bigl(1+\frac{|z-w|}{|w|}\Bigr)}
$$
$$
=1+\frac{\log\Bigl(1+\frac{|a||z-w|}{1-|a|(1-|w|)}\Bigr)}{\log\Bigl(1+\frac{|z-w|}{|w|}\Bigr)}
\le
1+\frac{\log\Bigl(1+\frac{|a|}{1-|a|(1-|w|)}\Bigr)}{\log\Bigl(1+\frac{1}{|w|}\Bigr)},
$$
where the last inequality follows from Lemma \ref{l3}, part 2.,
\textcolor{black}{ with
$$
B=|a|, \ C=|a|(1-|w|), \ D=|a||w|, \ \theta=|z-w|,
$$}
since $C+D=|a|<1$ and $|z-w|\le|z|+|w|\le1$.

Denote now $|w|=r$ and let $k(r)=k_1(r)/k_2(r)$ with
$$
k_1(r)=\log\Bigl(1+\frac{|a|}{1-|a|(1-r)}\Bigr); \
k_2(r)=\log\Bigl(1+\frac{1}{r}\Bigr).
$$

We shall show now that the function $k(r)$ is monotone increasing
on the positive part of real axis.

 Indeed, since $k_1(\infty)=k_2(\infty)=0$ and

$$
k_1'(r)/k_2'(r)=\frac{|a|^2r(1+r)}{(1+|a|r)(1-|a|+|a|r)}=\frac{|a|(1+r)}{1+|a|r)}\frac{|a|r}{1-|a|+|a|r}
$$
$$
=(1-\frac{1-|a|}{1+|a|r})(1-\frac{1-|a|}{1-|a|+|a|r}),
$$
with both functions in parenthesis evidently increasing on
$\mathbb R^+$, the conclusion follows from Lemma \ref{Le2}.

Since in this case $0<r\le 1/2$, we also obtain that

$$
J(z,w;a)\le1+\frac{\log\Bigl(1+\frac{|a|}{1-|a|/2}\Bigr)}{\log\Bigl(1+2\Bigr)}=C_2(a).
$$

\bigskip

The constant $C_2(a)$ is sharp since
$J(\frac{a}{2|a|},\frac{-a}{2|a|};a)=C_2(a)$.

\vskip 1cm

Because $C_2(a)=1+(\log\frac{2+a}{2-a})/\log 3>\frac{1}{1-\frac
{|a|}{2}}=C_1(a)$, we conclude that the best possible upper bound
$C$ is $C=C_2(a) \hfill \square$.


\bigskip

Finally, in order to widen the topic started with Conjecture 1,
we consider the following:

Let $h:\mathbb D\to\mathbb D$ be M\"obius map with $h(0)=a$. A
challenging problem is to determine best possible j-Lip constants
$C(m,a)$ such that
$$
j(h^m(z), h^m(w))\le C(m,a) j(z,w),
$$
for all $z,w\in\mathbb D$ and $ m\in\mathbb N$.

\bigskip

It is not difficult to show that $C(m,a)\le 1+|a|=C(1,a), \ m\in
\mathbb N$. Therefore, the following question naturally arise.

\bigskip

Q1. {\em Is the sequence $C(m,a)$  monotone decreasing in $m$?}

\bigskip

A partial answer is given in the next

\begin{theorem}\label{s2}
The sequence $C(2^n,a), n\in\mathbb N$ is monotone decreasing in
$n$.
\end{theorem}

\bigskip

\begin{proof}
Indeed, since

$$
\frac{|h^{2^{n+1}}(z)-h^{2^{n+1}}(w)|}{1-\max\{|h(z)|^{2^{n+1}},|h(w)|^{2^{n+1}}\}}=\frac{|h^{2^n}(z)+h^{2^n}(w)|}{1+\max\{|h(z)|^{2^n},|h(w)|^{2^n}\}}
\frac{|h^{2^n}(z)-h^{2^n}(w)|}{1-\max\{|h(z)|^{2^n},|h(w)|^{2^n}\}}
$$
$$
\le
\frac{2\max\{|h|^{2^n}(z),|h|^{2^n}(w)\}}{1+\max\{|h(z)|^{2^n},|h(w)|^{2^n}\}}
\frac{|h^{2^n}(z)-h^{2^n}(w)|}{1-\max\{|h(z)|^{2^n},|h(w)|^{2^n}\}}\le
\frac{|h^{2^n}(z)-h^{2^n}(w)|}{1-\max\{|h(z)|^{2^n},|h(w)|^{2^n}\}},
$$

we conclude that

$$
j(h^{2^{n+1}}(z), h^{2^{n+1}}(w))\le j(h^{2^n}(z), h^{2^n}(w)),
$$

i.e.,

$$
C(2^{n+1},a)=\sup_{z,w\in\mathbb D}\frac{j(h^{2^{n+1}}(z),
h^{2^{n+1}}(w))}{j(z,w)}\le \sup_{z,w\in\mathbb
D}\frac{j(h^{2^n}(z), h^{2^n}(w))}{j(z,w)}=C(2^n,a).
$$

\end{proof}

\bigskip

Q2. {\em  Is it true that $C(m,(m+1)^{-1})=1$ for $m\ge 2$?}

\bigskip

\bigskip
{\bf Acknowledgement.} The research of the second author was supported by the Academy of Finland grant with the
Project number 2600066611.


\end{document}